\documentclass[a4paper,12pt]{article}

\usepackage{amssymb, amsmath}
 \usepackage{amsthm}

\title{Wiman-Valiron method for fractional derivatives and sharp  growth estimates of $\alpha$-analytic solutions for linear fractional differential equations}
 \author{Igor Chyzhykov}

\newcommand{\f}{\varphi}
\newcommand{\RR}{\mathbb{R}}
\newcommand{\NN}{\mathbb{N}}
\newcommand{\ZZ}{\mathbb{Z}}
\newcommand{\D}{\mathbb{D}}

\newtheorem{theorem}{Theorem}
\newtheorem{lemma}{Lemma}
\newtheorem{remark}{Remark}
\newtheorem{example}{Example}
\newtheorem*{cor*}{Corollary}

\newtheorem{lettertheorem}{Theorem}

\numberwithin{equation}{section}

\begin{document}
\maketitle
\begin{abstract}
We consider a fractional linear differential equation with successive derivatives  given by
$	\mathbb{D}_\alpha^{n}y+ p_{n-1}(x) \mathbb{D}_\alpha^{n-1}y+ \dots +p_{1}(x)\mathbb{D}_\alpha y+p_0(x)y=0$, where $\mathbb{D}_\alpha^{j}$ is the $j$th iteration of the Caputo-Djrbashian fractional derivative of order $\alpha>0$, $p_j$ are $\alpha$-analytic functions for $0<x^\alpha <R$. Generalizing a result of Kilbas, Rivero, Rodr\'iguez-Germ\'a and Trujillo, we prove the existence and  uniqueness of the corresponding Cauchy problem in the class of $\alpha$-analytic functions.  We establish an exact growth order for the solution  when $p_j(x)=P_j(x^\alpha)$, where $P_j$ are polynomials, and  $p_0$ dominates in some sense. This is the full counterpart of the classical case of ordinary differential equations.  In particular, we demonstrate the sharpness of Kochubei's result and generalize it. To achieve this, we extend the Wiman-Valiron theory to analytic functions and the Djrbashian-Gelfond-Leontiev generalized fractional derivatives.


MathSubjClass:  34A08 (primary),  26A33, 30B10, 33E12, 34A12, 34A25, 34M03.
\end{abstract}

\section{Introduction}
\setcounter{section}{1} \setcounter{equation}{0} 

\subsection{\bf Fractional linear differential equations}
Though fractional differential equations have a variety of applications in modelling physical processes, their theory is not as well developed as that of ordinary differential equations.

Let $\alpha>0$. This paper studies sequential fractional linear differential equations of the form (see  \cite[Chap. V.1]{MiRo})
\begin{equation}\label{e:frac_lin_eq}
	\mathbb{D}_\alpha^{n}y+ p_{n-1}(x) \mathbb{D}_\alpha^{n-1}y+ \dots +p_{1}(x)\mathbb{D}_\alpha y+p_0(x)y=0,
\end{equation}
where $\mathbb{D}_\alpha$ is the  fractional Caputo-Djrbashian derivative, $\mathbb{D}_\alpha^{j}$ its $j$th iteration,  $p_j$, $j\in\{0, \dots, n-1\}$ are $\alpha$-analytic functions, i.\ e.\  $p_j(x)=P_j(x^\alpha)$, $P_j(z)$ being analytic on $\{z\in \mathbb{C}: |z|<R\}$, $0<R\le \infty$.  
This type of equation was introduced and studied in the case of constant coefficients and the Riemann-Liouville operator  $D^\alpha$ instead of the Caputo-Djrbashian operator  in \cite[Chap.\ V]{MiRo}. It is worth remarking that neither the semigroup property $D^\alpha D^\beta = D^{\alpha + \beta}$ nor the commutativity property $D^\alpha D^\beta = D^\beta D^\alpha$ holds in general for either the Riemann-Liouville or the Caputo-Djrbashian operators (see \cite[Chap. IV.6]{MiRo}, \cite{Chikriy_mat}). In particular,  $\mathbb{D}_\alpha^{j}\ne \mathbb{D}_{\alpha j}$. Suffiecient conditions for the semigroup property and commutativity can be found in \cite[Chap. IV.6]{MiRo}, \cite{BegCap}.

Note that the existence and uniqueness of a solution to the  Cauchy problem for a more  general than \eqref{e:frac_lin_eq} linear  differential equation of fractional order  was established in  \cite{DjNe}.

As it is stated in \cite{KRRT} for the cases $n=1$ and $n=2$ equation  \eqref{e:frac_lin_eq} possesses the unique $\alpha$-analytic   soultion on $0<|x|^\alpha<R$ that satisfy initial conditions for both the Riemann-Liouville  and the Caputo-Djrbashian derivatives. Although the formal series representations of solutions are given in \cite{KRRT}, their convergence is not rigorously proved. 
On the other hand, A. Kochubei  (\cite{Koch09}) established an asymptotically sharp estimate of the growth of solutions to equation \eqref{e:frac_lin_eq} in the case $n=1$ where $A$ is a polynomial. In other words, {\sl the Cauchy problem
	$$ \mathbb{D}_\alpha y +a(x)y=0, \quad y(0)=y_0,$$
	where $a(x)=A(x^\alpha)$, and $A$ is a polynomial of degree $m$, has a unique solution of the  form $y(x)=v(x^\alpha)$, where $v$ is an entire function of order not greater than $\frac{m+1}{\alpha}$.}

In the classical case, when $\alpha = 1$, there are various sharp estimates for the growth of solutions in both model cases when the coefficients are entire functions or analytic in the unit disc $\{z\in \mathbb{C}: |z|<1\}$ (see, for example,  \cite{GG88}, \cite{GSW}, \cite{Laine}, \cite{CGH}, \cite{CHR_anal}, and  \cite{CHR_aus}). Two principal tools are used to obtain lower estimates for the growth of solutions: the Wiman-Valiron method, which was originally developed in  \cite{Wi14,Wi16,Val14,Val18} (see also the survey \cite{Hay74}) and the logarithmic derivative estimate (\cite{GG88}, \cite{CGH}, \cite{CHR_aus}).  

On one hand, there is still no understanding of how one can generalize the logarithmic derivative estimate approach for fractional derivatives (cf.\ \cite{ChSem1}). On the other hand, in \cite{ChSem2} the authors succeeded in generalizing the Wiman-Valiron method for  Riemann-Liouville derivatives. This allowed us to obtain sharp asymptotic growth of solutions to a special fractional linear differential equation, but not for \eqref{e:frac_lin_eq}. The reason is that, to find an asymptotic for an $\alpha$-analytic solution  of \eqref{e:frac_lin_eq},  one needs to generalize the Wiman-Valiron  theory for Gelfond-Leontiev type derivatives (see \cite{Kirjak94}, \cite{Kirjak99}). We do this in Section \ref{s:proof}, where we prove Theorem \ref{th:frwv},  the main result of the paper.  In the final section we study equation \eqref{e:frac_lin_eq}.   Theorem \ref{t:Cauchy} establishes the existence and uniqueness of a solution to the Cauchy problem for  \eqref{e:frac_lin_eq} in the class of $\alpha$-analytic functions. This is a slight generalization of a result from \cite{KRRT}.  We then consider the case when all coefficients of \eqref{e:frac_lin_eq} are polynomials of $t^\alpha$. We  prove (Theorem \ref{t:sol_finite_order}) that in this case 
all $\alpha$-analytic solutions are of the form $v(t^\alpha)$ where $v$ is an entire function of finite order of the growth.
Finally, Theorem \ref{t:sharp_order} gives sharp values for the order of the growth of $v$ under natural conditions on the coefficients. 
This improves the mentioned Kochubei's result as a special case (see the corollary in Section \ref{s:final}). 
Auxiliary results are given in Section \ref{s:prelim}.

We use the notation $a\lesssim b$ if there exists a constant $C>0$ such that $a\le Cb$. Similarly, $a\gtrsim b$ is understood in an analogous manner. If $a\lesssim b$ and $a\gtrsim b$, then we write $a\asymp b$ and say that $a$ and $b$ are comparable. Additionally, $a(t)\sim b(t)$ means that the quotient $a(t)/b(t)$ approaches one as $t$ tends its limit. By $[x]$ we denote the entire part of $x\in \mathbb{R}$. 

\subsection{\bf Fractional integrals and derivatives}
Let $0<T\le \infty$ and  $L(0,T)$ be the class of all integrable functions on $(0,T)$. The Riemann-Liouville fractional derivative of order $\alpha>0$ for $\f\in L(0,T)$ is defined as
$$D^{\alpha}\f(x)=\frac{d^{n}}{dx^{n}}\{I^{n-\alpha}\f(x)\}, \quad \alpha\in(n-1,n],\quad n\in\mathbb{N},$$
where
$$I^{\alpha}\f(x)=\frac{1}{\Gamma(\alpha)}\int\limits_{0}^{x}\frac{\f(t)\,dt}{(x-t)^{1-\alpha}}$$
is the Riemann-Liouville fractional integral of order $\alpha>0$ for $\f$, $\Gamma(\alpha)$ is the Gamma function. In particular, if $0<\alpha<1$, then
$$D^{\alpha}\f(x)=\frac{1}{\Gamma(1-\alpha)}\frac{d}{dx}
\int\limits_{0}^{x}\frac{\f(t)\, dt}{(x-t)^{\alpha}},$$
provided that $I^{1-\alpha} \f $ is absolutely continuous on $(0,T)$.

The fractional derivative and integral  have the following property (\cite{SaKiMa})
\begin{gather}\label{e:fracderpol}
	I^\alpha x^{\beta-1}= \frac{\Gamma (\beta)}{\Gamma(\beta+\alpha)}x^{\beta+\alpha-1}, \;	D^{\alpha}x^{\beta-1}=\frac{\Gamma(\beta)}{\Gamma(\beta-\alpha)}x^{\beta-\alpha-1},\quad\alpha,\beta>0, \alpha\ne \beta, \\
	D^\alpha 1= \frac{1}{\Gamma(1-\alpha)}x^{-\alpha}, \; D^\alpha x^{\alpha-j}=0, \quad \alpha>0, j\in \{1,2,\dots, [\alpha]+1\}.
\end{gather}
We can see that, on  one hand, Riemann-Liouville fractional differentiation can produce a singularity and, on the other hand, it can be defined on functions with a singularity at the origin. Despite this the \emph{Caputo-Djrbashian, or regularized fractional derivative} 
\begin{gather}\nonumber
	(\D^\alpha \f) (x)=D^\alpha\left( \f(x)-\sum_{k=0}^{n-1} \frac{\f ^{(k)}(0)}{k!}x^{k}\right)\\
	= D^\alpha  \f(x)-\sum_{k=0}^{n-1} \frac{\f^{(k)} (0)}{\Gamma (k+1-\alpha)}x^{k-\alpha}, \quad n-1<\alpha\le n,\label{e:cap_djr_der}
\end{gather}
is defined on functions that are continuous with their derivatives up to order $n-1$  and vanishes  on constants,   which is more natural for physical applications.

Though the operator $	I^\alpha $ is associative and commutative  with respect to the index, i.e. $I^\alpha \circ I^\beta= I^\beta \circ I^\alpha=I^{\alpha+\beta}$, $\alpha, \beta>0$, this is not the case for $D^\alpha$ and $\mathbb{D}_\alpha$ (see \cite[Chap.IV]{MiRo}, \cite{BegCap}).

\begin{example} \label{ex:1}
	Let $\alpha=\frac 12$, $u(t)=u_0+u_1t^{\frac 12}+ u_2t$, $t>0$. Then \begin{gather*}
		\mathbb{D}_{\frac 12} u(t)= u_1 \frac{\Gamma(3/2)}{\Gamma(1)} +u_2\frac{\Gamma(2)}{\Gamma(3/2} t^{\frac 12}, \\
		\mathbb{D}_{\frac 12} ^2 u(t)= u_2, \\
		\mathbb{D}_{1} u(t)=u'(t)= \frac 12 u_1t^{-\frac 12} +u_2.
	\end{gather*}
\end{example} 

Let $H_\alpha(R)$, $0<R\le \infty$, $\alpha>0$, denote the class of $\alpha$-analytic functions on $(0, R^{1/\alpha})$, that is the functions $u$ represented in the form  $u(t)=\sum_{m=0}^\infty u_m t^{\alpha m}$, $\alpha>0$, $ 0\le t^\alpha<R\le \infty$. Direct computation yields
\begin{gather}\label{e:D_alpha}
	(D^\alpha u)(t)=\sum_{m=0}^\infty u_m \frac{\Gamma(m\alpha +1)}{\Gamma((m-1)\alpha +1)}t^{\alpha(m-1)},\quad 0< t^\alpha< R, \\
	\nonumber
	(\D_\alpha u)(t)=\sum_{m=0}^\infty u_m \frac{\Gamma(m\alpha +1)}{\Gamma((m-1)\alpha +1)}t^{\alpha(m-1)}-\frac{u_0}{\Gamma(1-\alpha)}t^{-\alpha}\\ =\sum_{m=1}^\infty u_m \frac{\Gamma(m\alpha +1)}{\Gamma((m-1)\alpha +1)}t^{\alpha(m-1)}.
	\label{e:Cap_alpha}
\end{gather}

\begin{remark}
	If $u(0)=0$, and $u$ is $\alpha$-analytic for $ 0<t^\alpha<R$, then $D^\alpha u(t)=\D_\alpha u(t)$, and  $D^\alpha u(t)$ is $\alpha$-analytic for $ 0<t^\alpha<R$ as well.
\end{remark}

\begin{remark} \label{r:2} Repeating the argument from the previous remark, we have that
	\begin{equation}
		(\D_\alpha^j u)(t)=\sum_{m=j}^\infty u_m \frac{\Gamma(m\alpha +1)}{\Gamma((m-j)\alpha +1)}t^{\alpha(m-j)}=\sum_{m=0}^\infty u_{m+j} \frac{\Gamma((m+j)\alpha +1)}{\Gamma(m\alpha +1)}t^{\alpha m}.
		\label{e:Cap_j}
	\end{equation}	
	In particular, if $u_0=\dots= u_{j-1}=0$, then $(D^\alpha)^j u(t)=\D_\alpha^j u(t)$.
\end{remark}
%

\begin{lemma} \label{l:commute_d_al}
	If $u\in H_\alpha(R)$,  then $D^\alpha$ is associative and commutative providied   
	the conditions 	 $u_m=0$ for $m<\frac{\gamma+\beta}{2}$, $\gamma, \beta>0$,  i.e.
	$$ D^\beta \circ D^\gamma =D^\gamma \circ D^\beta =D^{\beta+\gamma}.$$
	In particular, $\left(D^\alpha\right)^j =D^{j\alpha}$ if $u_m=0$ for $m<j$. 
\end{lemma}

\begin{proof}[Proof of Lemma \ref{l:commute_d_al}]
	We write $u(t)=\sum_{m=0}^\infty u_m t^{\alpha m}$, $\alpha>0$, $ 0<t^\alpha<R\le \infty$. Since the power series is uniformly convergent on every segment $[0, r^\alpha] \subset [0, R )$ (cf. \cite[Theorem 3, Sec. IV.6]{MiRo}), we can integrate and differentiate it under the sum sign  at every point of $[0, R^{1/\alpha} )$. Then using the fact that $u_m=0$ for $m < (\gamma+\beta)/\alpha $ we obtain
	\begin{gather*} D^\beta D^\gamma \sum_{m=0}^\infty u_m t^{\alpha m}=D^\beta \sum_{m=\left[\frac{\gamma+\beta}{\alpha}\right]}^\infty  D^\gamma\left(  u_m t^{\alpha m}\right) \\= D^\beta \sum_{m=\left[\frac{\gamma+\beta}{\alpha}\right]}^\infty  u_m  \frac{\Gamma(\alpha m+1)}{\Gamma(\alpha m+1-\gamma)}t^{\alpha m-\gamma}.
	\end{gather*}
	Since $ \frac{\Gamma(\alpha m+1)}{\Gamma(\alpha m+1-\gamma)}\sim (\alpha m)^\gamma$, $m\to \infty$, the power series under the operator $D^\beta$ has the same radius of convergence. Then 
	\begin{gather} D^\beta D^\gamma \sum_{m=0}^\infty u_m t^{\alpha m} \nonumber
		\\= \sum_{m=\left[\frac{\gamma+\beta}{\alpha}\right]}^\infty  u_m  \frac{\Gamma(\alpha m+1)}{\Gamma(\alpha m+1-\gamma)} \frac{\Gamma(\alpha m-\gamma+1)}{\Gamma(\alpha m+1-\gamma-\beta)}t^{\alpha m-\gamma-\beta}=D^{\gamma+\beta} u(t).
	\end{gather}
	The equality $D^\gamma \circ D^\beta =D^{\beta+\gamma}$ follows by exchanging  the roles of $ \beta $ and $\gamma$. 
\end{proof}

Let
\begin{equation}\label{e:entfun}f(z)=\sum\limits_{n=0}^{\infty}a_{n}z^{n},\quad z=re^{i\theta}
\end{equation}
be an entire function. 
Let $u(t)=f(t^\alpha)=\sum_{n=0}^\infty a_n t^{\alpha n}$, $ t>0$. Direct computation shows that
\begin{equation}\label{e:GL_d_alpha}
	(\D^\alpha u)(t)=\sum_{n=1}^\infty a_n \frac{\Gamma(n\alpha +1)}{\Gamma(n\alpha +1-\alpha)}t^{\alpha(n-1)}= (\mathcal{D}^\alpha f)(t^\alpha),\end{equation}
where $$ 
(\mathcal{D}^\alpha f)(z)=\sum_{n=1}^\infty a_n \frac{\Gamma(n\alpha +1)}{\Gamma(n\alpha +1-\alpha)}z^{n-1},$$
is the  so-called Djrbashian-Gelfond-Leontiev operator (\cite{Kirjak94}, \cite{Kirjak99}, \cite{SaKiMa}), a special case of the Gelfond-Leontiev generalized differential operator corresponding to the Mittag-Leffler function $E_\alpha(z)=\sum_{k=0}^\infty \frac{z^k}{\Gamma(k\alpha+1)}$, $\alpha>0$.

The corresponding integral operator, right inverse to $\mathcal{D}_\alpha$,  can be written as
$$ 
(\mathcal{I}_\alpha f)(z)=\sum_{n=0}^\infty a_n \frac{\Gamma(n\alpha +1)}{\Gamma(n\alpha +1-\alpha)}z^{n+1},$$
with $ \mathcal{I}_\alpha \mathcal{D}_\alpha f (z)=f(z)-f(0)$.
We also have \cite[Sec. 22.3]{SaKiMa} the integral representation
$$ \mathcal{I}_\alpha f(z)=\frac{z}{\Gamma(\alpha)} \int_0^1 (1-t)^{\alpha-1} f(zt^\alpha)\, dt.$$
It follows from \eqref{e:GL_d_alpha} that  (\cite[Section 2.5]{KST}, \cite[Sec. 18.2, 22.3]{SaKiMa}) 
$$ 
\mathcal{D}_\alpha =Q\circ \D_\alpha \circ Q^{-1},$$
where $Q$ is the substitution operator $z \mapsto z^{\frac 1\alpha }$, $Q^{-1}$ is its inverse, for a corresponding branch of a multivalued power function chosen on the segment $[0, z]$.  This note allows us to apply the approach used in \cite{ChSem2} to derive Wiman-Valiron type results for the operator $\mathcal{D}_\alpha$.
\subsection{Wiman-Valiron theory}
For  $r\in [0,+\infty)$ and an entire function $f$ of the form \eqref{e:entfun} we  denote $M(r,f)=\max\{|f(z)|:|z|=r\}$. We define the maximal term as $\mu(r,f)=\max\{|a_{n}|r^{n}:n\geq0\}$ and the central index of the series as $\nu(r,f)=\max\{n\geq0:|a_{n}|r^{n}=\mu(r,f)\}$. Note that $\nu(r,f) $ is non-decreasing, and an entire function $f$ is transcendental if and only if $ \nu(r,f)\to +\infty$ as $r\to +\infty$.

For a non-constant entire function $f$, the order $\sigma(f)$ is defined as follows:
\begin{gather}\label{e:order}
	\sigma(f):=\varlimsup_{r\to \infty}\frac{\log \log M(r,f)}{\log r}=\varlimsup_{r\to \infty}\frac{\log \nu(r,f)}{\log r}.\end{gather}


Let $V$ be the class of positive continuous nondecreasing functions $v$ on $[0,+\infty)$ such that $\frac{x^{2}}{v(x)\ln v(x)}$ increases to $+\infty$ on $x\in[x_{0};+\infty)$, $x_{0}>0$, and $\int\limits_{0}^{+\infty}\frac{dx}{v(x)}<+\infty$. For example, the functions $v_0(x)=x^\beta$ $(x\ge  0)$, $1<\beta<2$, and   $v_1(x)=x\ln^{\alpha+1}x$, $(x\geq e)$,  $\alpha\in(0,1)$, belong to $V$.

A measurable set $E\subset [0, \infty) $ is called of finite logarithmic measure if $\int_{E\cap [1,\infty)}  \frac {dt}t <\infty$.

The main result of the Wiman-Valiron theory can be formulated as follows (cf. {\cite[Lemma 3.8]{Sher}})
\begin{lettertheorem}\label{th:mrwv}
	Let $v\in V$ and $\varkappa(t)=4\sqrt{v(t)\ln v(t)}$. Suppose that $f$ is an entire function, $|z_{0}|=r$, and 
	$$|f(z_{0})|\geq   M(r,f)  v^{-2}(\nu(r,f)),$$
	holds. There exists a set $E\subset \RR_+$ of finite logarithmic measure such that if
	$$r\left(1-\frac{1}{40\varkappa(\nu)}\right)<\rho<r\left(1+\frac{1}{40\varkappa(\nu)}
	\right),\quad r\not \in  E, \nu=\nu(r,f),$$
	and $q\in\mathbb{Z}_{+}$, then  we have for $|z|=\rho$
	$$\left(\frac{r}{\nu}\right)^{q}f^{(q)}(z)=
	f(z)+O\left(\frac{\varkappa(\nu)}{\nu}\right)M(\rho,f).$$
	In particular, if $\ln\rho-\ln r=o\left(\frac{1}{\varkappa(\nu)}\right)$ then
	$$M(\rho, f^{(q)})=\left(\frac{\nu}{\rho}\right)^{q}
	\left\{1+O\left(\frac{\varkappa(\nu)}{\nu}\right)\right\}M(\rho,f)=(1+o(1))\left(\frac{\nu}{r}\right)^{q}
	M(r,f)$$
	as $r\rightarrow+\infty$, $r\not \in E$.
\end{lettertheorem}

The main result of \cite{ChSem2} literally repeats Theorem \ref{th:mrwv} for arbitrary $q>0$ and the Riemann-Liouville derivative $D^q f$ instead of $f^{(q)}$. Note that $|z|^q D^q f$ is a single-valued function of~$z$. 

We generalize the Wiman-Valiron method for the Djrbashian-Gelfond-Leontiev fractional derivative $\mathcal{D}_\alpha$.

\begin{theorem}\label{th:frwv}
	Let $v\in V$ and $\varkappa(t)=4\sqrt{v(t)\ln v(t)}$. Suppose that $f$ is an entire function,  $|z_{0}|=r,$ and 
	$$|f(z_{0})|\geq M(r,f) v^{-2}(\nu(r,f))$$
	holds. Then there exists a set $E\subset \RR_+$ of finite logarithmic measure such that if $$r\left(1-\frac{1}{40\varkappa(\nu)}\right)<\rho<r\left(1+\frac{1}{40\varkappa(\nu)}
	\right),\quad r\not\in E, \nu=\nu(r,f),$$
	$\alpha>0$ and $j\in \mathbb{N}$, then  we have for $|z|=\rho$
	\begin{equation}\label{e: th1} \mathcal{D}_{\alpha}^jf(z)=
		(\nu \alpha)^{j\alpha} \left(\frac{f(z)}{z^j}+O\left(\frac{\varkappa(\nu)}{\nu}\right)\frac{M(\rho,f)}{\rho^j}\right).
	\end{equation}
	In particular, if $\ln\rho-\ln r=o\left(\frac{1}{\varkappa(\nu)}\right)$ then
	\begin{equation}\label{e: th2}M(\rho, \mathcal{D}_{\alpha}^j f(z))=\frac{(\nu\alpha)^{j\alpha}}{\rho^j}
		\left\{1+O\left(\frac{\varkappa(\nu)}{\nu}\right)\right\}M(\rho,f)=(1+o(1))\frac{(\nu \alpha)^{j\alpha}}{r^j}
		M(r,f)
	\end{equation}
	as $r\rightarrow+\infty$, $r\not \in E$.
\end{theorem}

\section{Preleminaries} \label{s:prelim}
\setcounter{section}{2} \setcounter{equation}{0} 

\subsection{Auxiliary results from Wiman-Valiron theory}
To prove Theorem \ref{th:frwv} we need the following statements frequently used in the Wiman-Valiron theory.
For $\rho\in[0;+\infty)$ we put $\mu(r,\rho,f)=|a_{\nu(r,f)}|\rho^{\nu(r,f)}$.

\begin{lemma}[{\cite[Lemma 3.4]{Sher}, cf.  \cite[Lemma 2]{Hay74}}]\label{l:lem 3.4.}Let $v\in V$ and $\varkappa(t)=4\sqrt{v(t)\ln v(t)}$. Then for any fixed positive $q$ and for all $\rho$, $|\ln\rho-\ln r|\leq\frac{1}{\varkappa(\nu)},$  we have
	\begin{equation}\label{e: estrterm}\sum\limits_{|n-\nu|>\varkappa(\nu)}n^{q}|a_{n}|\rho^{n}=
		o\left(\frac{\nu^{q}\mu(r,\rho,f)}{v(\nu)^{3}}\right), \quad \nu=\nu(r,f),
	\end{equation}
	as $r\rightarrow+\infty$ outside a set of finite logarithmic measure.
\end{lemma}
\begin{lemma}[{\cite[Lemma 3.5]{Sher}, cf.  \cite[Lemma 7]{Hay74}}]\label{l: lem3.5.}Suppose that $P$ is a polynomial of degree $m$ and $|P(z)|\leq M$ for $|z|\leq r$. Then for $R\geq r$ we have
	$$|P'(z)|\leq\frac{eMmR^{m-1}}{r^{m}}, \quad |z|<R.$$
\end{lemma}
\begin{lemma}[{\cite[Lemma 3.6]{Sher}, cf.  \cite[Lemma 8]{Hay74}}]\label{l: lem3.6.}Suppose that $P$ is a polynomial of degree $m$ and $|P(z)|\leq M$ for $|z|<r$. If $|z_{0}|\leq r$ and $|P(z_{0})|\geq\eta M,$ $0<\eta\leq1$, then for $|z-z_{0}|\leq\frac{\eta  r}{8m}$ we have
	$$\frac{1}{2}|P(z_{0})|\leq|P(z)|\leq\frac{3}{2}|P(z_{0})|.$$
\end{lemma}
\begin{lettertheorem}[{\cite[Lemma 3.7]{Sher}, cf.\  \cite[Theorem 10]{Hay74}}]\label{l: lem3.7.}Let $v\in V$ and $\varkappa(t)=4\sqrt{v(t)\ln v(t)}$. Suppose that $f$ is an entire function, $|z_{0}|=r,$ $r\not\in E$ a set of finite logarithmic measure,
	$$|f(z_{0})|\geq\eta M(r,f),\quad v^{-2}(\nu(r,f))\leq\eta\leq1.$$
	Then, if $z=z_{0}e^{\tau}$, $|\tau|\leq\frac{\eta}{18\varkappa(\nu)}$, $\nu=\nu(r,f)$, we have
	$$\ln\frac{f(z)}{f(z_{0})}=(\nu(r,f)+\varphi_{1})\tau+\varphi_{2}\tau^{2}+\delta(\tau),$$
	where
	$$|\varphi_{j}|\leq2,2\left(\frac{18\varkappa(\nu)}{\eta}\right)^{j},\quad (j=1,2),\quad |\delta(\tau)|\leq8,8\left(\frac{18\varkappa(\nu)\tau}{\eta}\right)^{3}.$$
\end{lettertheorem}

\subsection{Chain rule for higher derivaitves and Bell polynomials}
Suppose that $f\circ g$ is well-defined, and there exist $f^{(n)}$ and $g^{(n)} $ at the corresponding points. 
According to  Fa\'a di Bruno's formula 
$$
(f\circ g)^{(n)}=\sum_{j_1+2j_2+\dots+nj_n=n} \frac{n!}{j_1! \cdots j_n!} f^{(j_1+\dots+j_n)} (g) \prod_{s=1}^{n} \left(\frac{g^{(s)}}{s!}\right)^{j_s}.
$$

Noting that $j_s$ is zero for $s>n-k+1$ and combining the terms with the same values of $j_1+j_2+\dots+j_{n-k+1}=k$ we arrive to the formula
\begin{equation}\label{e:faa_di_bruno}
	(f\circ g)^{(n)}=\sum_{k=1}^n f^{(k)}(g) B_{n,k} (g', g'', \dots, g^{(n-k+1)})
\end{equation}
where 
\begin{equation}\label{e:Bell}
	B_{n,k}(z_1, \dots, z_{n-k+1})=\sum_{\substack {j_1+2j_2+\dots+(n-k+1)j_{n-k+1}=n \\ j_1+j_2+\dots+j_{n-k+1}=k}} \frac{n!}{j_1! \cdots j_{n-k+1}!}   \prod_{s=1}^{n-k+1} \left(\frac{z_s}{s!}\right)^{j_s},	\end{equation}
are called \emph{incomplete Bell polynomials}. For example,
$$ B_{n,n}(z_1)=z_1^n, \quad B_{n,n-1}(z_1, z_2)=\frac{n(n-1)}{2} z_1^{n-2}z_2, \text{ and } B_{n,1}(z_1, \dots, z_n)=z_n.$$  It is convenient to define $B_{n,0}\equiv 0$. 

We write $B^*_{n,k} (z_1, \dots, z_n):=B_{n,k} (|z_1|, \dots, |z_n|)$.
To estimate the Bell polynomial we need the following lemma. 
\begin{lemma} \label{l:Faa_di_Bruno_est} For $\alpha>0$, $n, k\in \mathbb{N}$, $n\ge k$, and $g(w)=w^\alpha$, we have
	\begin{equation}\label{e:B_nk_est}
		B_{n,k}^* (g', g'', \dots, g^{(n-k+1)})\le (n-1)! \alpha^k (\alpha+1)\cdots (\alpha+n-k) |w|^{k\alpha-n}.
	\end{equation}	
\end{lemma}
\begin{proof}[Proof of Lemma \ref{l:Faa_di_Bruno_est}]
	Since $g^{(j)}(w)=\alpha(\alpha-1)\cdots (\alpha-j+1)w^{\alpha-j}$, we have \begin{gather*}
		B_{n,n}(g')= \alpha^n w^{n\alpha-n}, \; n\in \mathbb{N}, \\
		B_{n, n-1}(g', g'')=\frac{n(n-1)}{2} \alpha^{n-1} (\alpha-1)w^{(n-1)\alpha-n}, \; n\ge 2.
	\end{gather*} 
	Thus, the assertion of the lemma holds for $k\in \{n-1, n\}$.
	
	To show that this is true for $1\le k<  n-1$ we use the induction in $k$.
	
	For $k=1$ $B_{n,1}(z_1, \dots, z_n)=z_n$, so $(w^\alpha)^{(n)}=\alpha(\alpha-1)\cdots (\alpha-n+1)w^{\alpha-n}$, and the assertion follows.
	
	Assume that \eqref{e:B_nk_est} holds for $n\le m$ and $1\le k\le n$, and $k\le m-1$. The case $k=m=n$ is already considered. Then
	\begin{gather*}
		(f\circ g)^{(m+1)}= \left( \sum_{k=1}^m f^{(k)}(g) B_{m,k} (g', \dots, g^{(m-k+1)}) \right)'\\
		= \sum_{k=1}^m \left(f^{(k+1)}(g) g' B_{m,k} (g', \dots, g^{(m-k+1)})+ f^{(k)}(g) (B_{m,k} (g', \dots, g^{(m-k+1)}))' \right) =\\
		= \sum_{k=2}^{m+1} f^{(k)}(g) g' B_{m,k-1} (g', \dots, g^{(m-k+2)})+ \sum_{k=1}^m f^{(k)}(g) (B_{m,k} (g', \dots, g^{(m-k+1)}))' .
	\end{gather*}
	Combining this with \eqref{e:Bell} we deduce
	\begin{gather}
		\nonumber 
		B^*_{m+1,k}(g', \dots, g^{(m-k+2)})=|g'| B^*_{m,k-1}(g', \dots, g^{(m-k+2)})+ \\
		+ \sum_{\substack {j_1+2j_2+\dots+(m-k+1)j_{m-k+1}=m \\ j_1+j_2+\dots+j_{m-k+1}=k}} \frac{m!}{j_1! \cdots j_{m-k+1}!} \sum_{s=1}^{m-k+1}  j_s \frac{|g^{(s+1)}|}{|g^{(s)}|}\prod_{l=1}^{m-k+1} \left(\frac{|g^{(l)}|}{l!}\right)^{j_l} \label{e:b_star_est}
	\end{gather}
	Evidently,  $\frac{|g^{(s+1)}|}{|g^{(s)}|}=\frac{|\alpha-s|}{|w|}$. Then, we have
	\begin{gather*}
		\sum_{s=1}^{m-k+1}  j_s \frac{|g^{(s+1)}|}{|g^{(s)}|} \\ \le \frac1{|w|}  (\alpha(j_1+\dots+ j_{m-k+1})+ j_1+2j_2+ \dots+ (m-k+j)j_{m-k+j})= \frac{\alpha k+m}{|w|}.
	\end{gather*}	
	
	We rewrite \eqref{e:b_star_est} as follows, using the induction assumption, 
	\begin{gather*}
		B^*_{m+1,k}(g', \dots, g^{(m-k+2)})\\ \le |g'| B^*_{m,k-1}(g', \dots, g^{(m-k+2)})+ \frac{\alpha k+m}{|w|} B^*_{m,k}(g', \dots, g^{(m-k+1)})\\
		\le \alpha |w|^{\alpha-1} (m-1)! \alpha ^{k-1} (\alpha+1)\cdots (\alpha+m-k+1) |w|^{(k-1)\alpha-m}\\ + 
		\frac{\alpha k+m}{|w|}  (m-1)! \alpha ^{k} (\alpha+1)\cdots (\alpha+m-k) |w|^{k\alpha-m}\\
		= (m-1)!  |w|^{k\alpha-m-1}  \alpha ^{k} (\alpha+1)\cdots (\alpha+m-k)  \left(  \alpha+m-k+1+\alpha k+m
		\right)\\
		\le m!  |w|^{k\alpha-m-1}  \alpha ^{k} (\alpha+1)\cdots (\alpha+m-k)(\alpha+m-k+1)
	\end{gather*}
	as long as $\alpha k+m\le (m-1) (\alpha+m-k+1)$.
	This inequality is equivalent to 
	$ \alpha(m-k-1)\ge m -(m-1)(m-k+1)$. However, the left-hand side is nonnegative  because $k\le m-1$, while the right-hand side is nonpositive for $m\ge 2$ because $\frac{m}{m-1}\le 2\le m-k+1$. The induction step is proved. The assertion of the lemma follows.
	
\end{proof}

\section{Proof of Theorem \ref{th:frwv}} \label{s:proof}
\setcounter{section}{3} \setcounter{equation}{0}

Let $(\nu=\nu(r,f))$ $$\nu_{1}=\min\{n: |n-\nu|\leq\varkappa(\nu)\},\quad \nu_{2}=\max\{n: |n-\nu|\leq\varkappa(\nu)\}.$$
By the definition of the class $V$, we have that $\nu/\varkappa(\nu) \uparrow +\infty$, so $\nu_1 \sim \nu_2 \sim \nu$ as $\nu \to+\infty$.

Since, by  Cauchy's inequality,  $\mu(r,\rho, f)\leq\mu(\rho,f)\le  M(\rho,f)$, from Lemma \ref{l:lem 3.4.} with $q=0$ for all $\rho$, $|\ln\rho-\ln r|\leq\frac{1}{\varkappa(\nu)}$, we obtain
\begin{equation}\label{e:estf}f(z)=P(z)z^{\nu_1}+o\left(\frac{\mu(r,\rho,f)}{v(\nu)^{3}}\right)=
	P(z)z^{\nu_{1}}+o\left(\frac{M(\rho,f)}{v(\nu)^{3}}\right),\quad |z|=\rho
\end{equation}
as $r\rightarrow+\infty$ outside a set $E$ of finite logarithmic measure, where
\begin{equation}\label{e:defP}P(z)=\sum\limits_{|n-\nu|\leq\varkappa(\nu)}
	|a_{n}|z^{n-\nu_{1}}.
\end{equation}

From (\ref{e:estf}) with $\rho=r$ we have $|P(z)|r^{\nu_{1}}\leq(1+o(1))M(r,f)$, $r\to\infty$, $r\not\in E$,  i.e. for all sufficiently large $r\not\in E$
\begin{equation}\label{e: estP}|P(z)|\leq\frac{1,01M(r,f)}{r^{\nu_{1}}}=:M^{*}(r),\quad|z|=r.
\end{equation}

We write
\begin{equation}\label{e: p+r} f(z)=P(z)z^{\nu_{1}}+R(z),
\end{equation}
where $P(z)$ is the polynomial (\ref{e:defP}). From now on we assume that $r$ is large enough so $\nu_1>\max\{j, \alpha j\}$.

Next, we need the asymptotic estimate of the  Gamma function (\cite{WW})
\begin{equation}\label{e: est gamma} \frac{\Gamma(t+a)}{\Gamma(t+b)}=t^{a-b}\left(1+O\left(\frac{1}{t}\right)\right),\quad t\rightarrow+\infty,\quad b, a\in \mathbb{R}.
\end{equation}

First, we estimate the fractional derivative of order $\alpha$ for $R(z)$. From Remark \ref{r:2} and Lemma \ref{l:lem 3.4.} we have
\begin{gather}\nonumber
	|\mathcal{D}_{\alpha}^j R(z)|=\left|
	\sum\limits_{|n-\nu|>\varkappa(\nu), n>j}\frac{\Gamma(1+n\alpha)}{\Gamma(1+n\alpha-j\alpha)}a_{n}\rho^{n-j}e^{in\theta}\right|\\ \leq C\sum\limits_{|n-\nu|>\varkappa(\nu)}(n\alpha)^{j\alpha}|a_{n}|\rho^{n-j}
	=o\left(\frac{\nu^{j\alpha}\mu(r,\rho,f)}{\rho^j v(\nu)^{3}}\right),\label{e: estfrdrt} \quad r\to\infty, r\not\in E.
\end{gather}
where $E$ is a set of finite logarithmic measure,  $C=\sup\limits_{n}\{2, \frac{\Gamma(n+1)}{\Gamma(n\alpha +1-j\alpha)}n^{-j\alpha}\}$.

Repeated application of Lemma \ref{l: lem3.5.} shows that for any $q\in\mathbb{Z}_{+}$ and $|z|=\rho$ we have that
\begin{equation}\label{e:estderP}|P^{(q)}(z)|=O\left(\left(\frac{\varkappa(\nu)}{r}\right)^{q}M^{*}(r)\right).
\end{equation}

In fact,
\begin{gather*}
	|P'(z)|\leq \frac{e M^{*}(r)2 \varkappa(\nu)\rho^{\nu_{2}-\nu_{1}-1}}{r^{\nu_{2}-\nu_{1}}} \\ \leq\frac{2e M^{*}(r)\varkappa(\nu)}{\rho}\left(1+\frac{1}{40\varkappa(\nu)}\right)^{2\varkappa(\nu)}=
	O\left(\frac{\varkappa(\nu)}{r}M^{*}(r)\right), \quad r\to\infty. \end{gather*}
Then 
\begin{gather*}
	|P^{(j)}(z)|\leq \frac{e M(\rho, P^{(j-1)})(2\varkappa(\nu)-j)
		\rho^{\nu_{2}-\nu_{1}-j-1}}{r^{\nu_{2}-\nu_{1}-j}}=O\left(\left(\frac{\varkappa(\nu)}{r}\right)^{j}M^{*}(r)\right).
\end{gather*}
We need the generalization Leibniz's formula
for fractional derivatives in order to estimate the fractional derivative of the first summand in (\ref{e: p+r}). Let $f(x)$ and $g(x)$ be  analytic functions on $[a,b]$. Then, according to (\cite[p.\ 216]{SaKiMa}),
\begin{equation}\label{e: form Leib} D^{q}(f\cdot g)=\sum\limits_{k=0}^{+\infty}{q\choose k}(D^{q-k}f)g^{(k)},
\end{equation}
where ${q\choose k}=\frac{(-1)^{k}q\Gamma(k-q)}{\Gamma(1-q)\Gamma(k+1)}$.

Taking into account Remark \ref{r:2} and Lemma \ref{l:commute_d_al} we deduce 
\begin{gather} \nonumber \mathcal{D}_\alpha^j (z^{\nu_{1}}P(z))=Q(\mathbb{D}_\alpha ^{j}(Q^{-1}(z^{\nu_1}P(z))))\\ = Q((D^\alpha)^{j} (w^{\alpha \nu_1} P(w^{\alpha})) )
	=Q(D^{\alpha j} (w^{\alpha \nu_1} P(w^{\alpha})) ), \quad |z|\to \infty. \label{e:associativity_D}
\end{gather}

In the following arguments we consider a branch of the power function chosen on the segment $[0,w]$ emanating from the origin. 
Using \eqref{e:associativity_D} and (\ref{e: form Leib}) we obtain
\begin{gather*}
	\mathcal{D}_\alpha^j (z^{\nu_{1}}P(z))
	= Q(D^{j\alpha} (w^{\alpha \nu_1} P(w^{\alpha})) )
	=Q\left( \sum\limits_{m=0}^{+\infty}{j\alpha\choose m}D^{j\alpha-m} w^{\alpha \nu_1}(P(w^\alpha))^{(m)}\right)\\
	=Q\left(\sum\limits_{m=0}^\infty {j\alpha\choose m}\frac{\Gamma(\alpha \nu_{1}+1)}{\Gamma(\alpha\nu_{1}+1-j\alpha+m)}w^{\alpha\nu_{1}+m-j\alpha} \right.\\ \left.\times \sum_{k=0}^m P^{(k)}(w^\alpha)B_{m,k}((w^\alpha)', \dots, (w^\alpha)^{(m-k+1)})\right)\\
	=\frac{\Gamma(\alpha\nu_{1}+1)}{\Gamma(\alpha\nu_{1}+1-j\alpha)}z^{\nu_{1}-j}P(z) \\ + Q\left(\sum\limits_{m=1}
	^{\infty}{j\alpha\choose m}
	\frac{\Gamma(\alpha\nu_{1}+1)}{\Gamma(\alpha\nu_{1}+1-j\alpha+m)}w^{\alpha\nu_{1}+m-j\alpha} \right. \\ \left. \times \sum_{k=1}^m P^{(k)}(w^\alpha)B_{m,k}((w^\alpha)', \dots, (w^\alpha)^{(m-k+1)}) \right)
	\\=:\frac{\Gamma(\alpha\nu_{1}+1)}{\Gamma(\alpha\nu_{1}+1-j\alpha)}z^{\nu_{1}-j}P(z)+ \tilde R(z).
\end{gather*}

Applying Lemma \ref{l:Faa_di_Bruno_est} and recalling that $\deg P\le 2\varkappa(\nu)$ we get
\begin{gather} \nonumber
	|	\tilde R(z)| \le Q \left(\sum\limits_{m=1}
	^{\infty}\left| {j\alpha\choose m}\right| 
	\frac{\Gamma(\alpha\nu_{1}+1)}{\Gamma(\alpha\nu_{1}+1-j\alpha+m)}|w|^{\alpha\nu_{1}+m-j\alpha} \right. \\  \nonumber
	\left. \times  \sum_{k=1}^m |P^{(k)}(w^\alpha)|\alpha^k (\alpha+1) \cdots 
	(\alpha+m-k)|w|^{k\alpha-m} \right)\\ \nonumber
	=\sum\limits_{m=1}
	^{\infty} \frac{j\alpha |\Gamma(m-j\alpha)|}{|\Gamma(1-j\alpha)| \Gamma(m+1)}
	\frac{\Gamma(\alpha\nu_{1}+1)}{\Gamma(\alpha\nu_{1}+1-j\alpha+m)}|z|^{\nu_{1}-j}\\ \nonumber 
	\times  \sum_{k=1}^m |P^{(k)}(z)|\alpha^k (\alpha+1) \cdots 
	(\alpha+m-k)|z|^{k} \\ =
	\frac{\alpha j \Gamma(\alpha\nu_1+1)|z|^{\nu_1-j}}{\Gamma (\alpha)|\Gamma(1-j\alpha)| } \sum_{k=1}^{2\varkappa(\nu)}  |P^{(k)}(z)|\alpha^k |z|^{k}
	\sum_{m=k}^ \infty   \frac{|\Gamma(m-j\alpha) |\Gamma(\alpha+m-k+1)}{\Gamma(m+1)\Gamma(\alpha \nu_1+1-j\alpha+m)}.  \label{e:Rz_est}
\end{gather}

Let $b_m=\frac{|\Gamma(m-j\alpha) |\Gamma(\alpha+m-k+1)}{\Gamma(m+1)\Gamma(\alpha \nu_1+1-j\alpha+m)} $. 
To estimate the sum $\sum_{m=k}^\infty$ we consider two cases. 
First, let $k\le m\le [2\alpha \nu_1]$. If $m\ge j\alpha$, then 
$$ \frac {b_{m+1}}{b_m} =\frac{m-j\alpha}{m+1} \frac{\alpha+m-k+1}{\alpha \nu_1+m+1-j\alpha} < \frac{2\alpha \nu_1 +\alpha-k+1}{3\alpha \nu_1+1-j\alpha} <\frac 34, \quad \nu\to \infty.$$
Otherwise, 
$$ \frac {b_{m+1}}{b_m} \le \frac{j\alpha}{2} \frac{\alpha+j\alpha+1}{\alpha \nu_1+2-j\alpha},$$
and we arrive to the same conclusion as $r\to \infty$ because $\nu_1\to \infty$.
So, 
\begin{gather}
	\label{e:b_m_segment} \sum_{m=k}^{[2\alpha \nu_1]} b_m < 4b_k =4\frac{|\Gamma(k-j\alpha) |\Gamma(\alpha+1)}{\Gamma(k+1)\Gamma(\alpha \nu_1+1-j\alpha+k)}.
\end{gather}
Second, if $m> 2\alpha \nu_1$, then using Stirling's formula \cite[Sec.\ 12.33]{WW} 
\begin{gather}
	\Gamma(x)=x^{x-\frac 12} e^{-x} \sqrt{2\pi} e^{\frac {\theta(x)}{12x}}, \quad \theta(x)\in (0,1), \; x\to+\infty, 
\end{gather}
we deduce $ (\theta_j(x)\in (0,1), j\in \{1,2\})$
\begin{gather*}
	\frac{\Gamma(\alpha+m-k+1)}{\Gamma(\alpha \nu_1+1-j\alpha+m)}\\  = \frac{(\alpha+m-k+1)^{\alpha+m-k+\frac 12}}{e^{\alpha+m-k+1}}\frac{e^{\frac{\theta_1}{12(\alpha+m-k+1)}}}{e^{\frac{\theta_2}{12(\alpha\nu_1+m-j\alpha+ 1)}}} \frac{e^{\alpha\nu_1+m-j\alpha+ 1}}{(\alpha\nu_1+m-j\alpha+ 1)^{\alpha\nu_1+m-j\alpha+ \frac 12}}\\
	=\frac{e^{\alpha \nu_1+k-(j+1)\alpha+o(1) }}{\left(\frac{\alpha \nu_1+m-j\alpha+1}{\alpha+m-k+1}\right)^{\alpha \nu_1+1+m-j\alpha}} \frac{(\alpha \nu_1+m-j\alpha+1)^{\frac 12}}{{(\alpha+m-k+1)^{\alpha \nu_1+k-(j+1)\alpha +\frac 12 }}}, \quad \nu \to+\infty.
\end{gather*}
Applying the inequality $e\le \left(1+\frac 1y\right)^{y+1}$, $y>1$ in the form  $e^\gamma\le \left(1+\frac \gamma x\right)^{x+\gamma}$, $\gamma<x$ with
$x=\alpha +m-k+1$ and $\gamma =\alpha\nu_1+k-(j+1)\alpha$, we obtain
\begin{equation} \label{e:b_m_large}
	\frac{\Gamma(\alpha+m-k+1)}{\Gamma(\alpha \nu_1+1-j\alpha+m)}\le 2 \frac{(\alpha \nu_1+m+1-j\alpha)^{\frac 12}}{{(\alpha+m-k+1)^{\alpha \nu_1+k-(j+1)\alpha +\frac 12 }}}.
\end{equation}

Thus, for $m> 2\alpha \nu_1$ we have that 
$$b_m \asymp \frac 1{m^{j\alpha+1}} \frac{m^{\frac 12}}{(m+o(1))^{\alpha \nu_1+k-(j+1)\alpha+\frac 12}}=\frac{1}{(m+o(1))^{\alpha \nu_1+k-\alpha+1}}, \quad \nu_1\to \infty.$$
Then  
\begin{gather}
		\label{e:b_m_tail1} 
	\sum_{m=[2\alpha \nu_1]+1}^{\infty} b_m \le \int_{2\alpha \nu_1} ^\infty \left(\frac 2x \right) ^{\alpha \nu_1+k-\alpha+1} \, dx= \frac{2}{\alpha \nu_1+k-\alpha} \frac 1{(\nu_1 \alpha)^{\alpha \nu_1+k-\alpha}}.
\end{gather}

We now show that the last term is infinitely small with respect to $b_k$ as $\nu\to \infty$. In fact, using Stirling's  formula we deduce that
\begin{gather*}
	\frac{1}{b_k (\alpha \nu_1+k-\alpha)} \frac 1{(\nu_1 \alpha)^{\alpha \nu_1+k-\alpha}}\\ \lesssim k^{j\alpha+1} \left(  \frac{\alpha\nu_1+k+1-j\alpha}{e}
	\right) ^{\alpha\nu_1+k+1-j\alpha} \frac {(\alpha\nu_1+k+1-j\alpha)^{-\frac 12} }{(\alpha\nu_1+k-\alpha)(\nu_1\alpha)^{\alpha\nu_1+k-\alpha}} \\
	\lesssim \frac{1}{(e+o(1))^{\alpha \nu_1+k-j\alpha}} \frac{k^{j\alpha+1}}{(\alpha\nu_1+k-j\alpha)^{\frac 12+(j-1)\alpha}}=o(1), \quad \nu\to \infty.
\end{gather*}
Combining this with \eqref{e:b_m_segment} and \eqref{e:b_m_tail1}, we obtain \begin{equation}\label{e:b_m_tail}
	\sum_{m=k} ^\infty b_m \lesssim b_k=\frac{|\Gamma(k-j\alpha)|}{\Gamma (k+1)}\frac{\Gamma(\alpha+1)}{\Gamma(\alpha\nu_{1}+1+k-j\alpha)}.
\end{equation}

Taking into account \eqref{e:Rz_est}, (\ref{e:estderP}) and (\ref{e: est gamma}) we obtain 
\begin{gather*}
	\left| \tilde R(z)\right| \lesssim |z|^{\nu_1-j}  \left|\sum\limits_{k=1}^{2\varkappa(\nu)} \frac{|\Gamma(k-j\alpha)|}{\Gamma (k+1)}\frac{\Gamma(\alpha\nu_{1}+1)}{\Gamma(\alpha\nu_{1}+1+k-j\alpha)} \alpha^k |z|^{k}P^{(k)}(z)\right|\\
	\lesssim |z|^{\nu_1-j} \sum\limits_{k=1}^{2\varkappa(\nu)} (\varkappa(\nu) \alpha)^k  \left(\frac{\rho}{r}\right)^{k}\frac{M^{*}(r)}{(\alpha \nu_1)^{k-j\alpha}}\\
	\lesssim  |z|^{\nu_1-j}(\alpha \nu_1)^{j\alpha}  \sum\limits_{k=1}^{2\varkappa(\nu)} \left(\frac{\varkappa(\nu)  \rho}{r \nu_1}\right)^{k}
	M^{*}(r)\lesssim |z|^{\nu_1-j}( \alpha\nu_1)^{j\alpha} \frac{\varkappa(\nu)}{\nu}M^{*}(r).
\end{gather*}

Therefore, in view of (\ref{e:estf}) and the previous estimate we have
\begin{gather}\nonumber
	\mathcal{D}^j_\alpha(f(z))=\frac{\Gamma(\alpha\nu_{1}+1)}{\Gamma(\alpha\nu_{1}+1-j\alpha)}z^{\nu_{1}-j}
	P(z)+O\left(\frac{\varkappa(\nu)}{\nu}\right)\rho^{\nu_1} \nu^{j\alpha} \frac{M^{*}(r)}{\rho ^j}
	\\=\frac{\Gamma(\alpha \nu_{1}+1)}{\Gamma(\alpha\nu_{1}+1-j\alpha)}\left(\frac{f(z)}{z^j}+
	o\left(\frac{\mu(r,\rho,f)}{\rho^j v(\nu)^{3}}\right)+O\left(\frac{\varkappa(\nu)}{\nu}\frac {M^{*}(r)}{\rho^j}
	\rho^{\nu_{1}}\right)\right).\label{e:estPz}
\end{gather}

Since $\int_0^\infty \frac {dt}{v(t)}<\infty$ and $v$ is nondecreasing, $v(t)/t \to+\infty$ as $t\to+\infty$. Hence $\frac{1}{v(t)^{3}}=o\left(\frac{\varkappa(t)}{t}\right)$, $t\rightarrow+\infty$, and using (\ref{e:estPz}) and (\ref{e: estfrdrt}) we obtain for $|z|=\rho$
\begin{gather}\nonumber
	\mathcal{D}^j_\alpha f(z)
	\\=\frac{\Gamma(\nu_{1}\alpha +1)}{\Gamma(\nu_{1}\alpha +1-j\alpha)}\left(\frac {f(z)}{z^j}+o\left(\frac{\varkappa(\nu)}{\nu}\frac{M(\rho,f)}{\rho^j} \right)+
	O\left(\frac{\varkappa(\nu)}{\nu}\frac{M(r,f)}{\rho^j} \left(\frac{\rho}{r}\right)^{\nu_{1}}\right)\right)\label{e:estzqfq} 
\end{gather}
when $r\rightarrow+\infty$ outside a set of finite logarithmic measure.

Next we choose $z_{0}$ so that $|f(z_{0})|=M(r,f)$ and take  $\tau=\ln(\rho/r)$, $\eta=1$. Then, by Theorem \ref{l: lem3.7.}, we have
$$\ln\left|f\left(\frac{\rho}{r}z_{0}\right)\right|=\ln|f(z_{0})|+\nu\tau+O(1), \quad|\tau|\leq\frac{1}{18\varkappa(\nu)},$$
so that
$$\ln M(\rho,f)\geq\ln M(r,f)+\nu\ln(\rho/r)+O(1).$$
Since $(\rho/r)^{\nu_{1}-\nu}=\exp\{\tau(\nu_{1}-\nu)\}=O(1)$, we deduce
$$\left(\frac{\rho}{r}\right)^{\nu_{1}}M(r,f)=
\left(\frac{\rho}{r}\right)^{\nu}\left(\frac{\rho}{r}\right)^{\nu_{1}-\nu} M(r,f)=O\left(\left(\frac{\rho}{r}\right)^{\nu}M(r,f)\right)=O(M(\rho,f)).$$

Thus, (\ref{e:estzqfq}) yields
\begin{equation}\label{e: estzqfq2}
	\mathcal{D}^j_\alpha f(z)=\frac{\Gamma(\nu_{1}\alpha +1)}{\Gamma(\nu_{1}\alpha +1-j\alpha)}
	\left(\frac {f(z)}{z^j}+O\left(\frac{\varkappa(\nu)}{\nu}\frac{M(\rho,f)}{\rho^j} \right)\right).
\end{equation}

From (\ref{e: est gamma}) we have
\begin{equation}\label{e: estfracgammma} \frac{\Gamma(\nu_{1}\alpha +1)}{\Gamma(\nu_{1}\alpha +1-j\alpha)}=
	(\nu\alpha)^{j\alpha}\left(1+O\left(\frac{1}{\nu}\right)\right),\quad \nu\rightarrow+\infty.
\end{equation}

Therefore, (\ref{e: estfracgammma}) implies
\begin{gather*}
	\mathcal{D}^j_\alpha f(z)=(\nu\alpha)^{j\alpha}\left(1+O\left(\frac{1}{\nu}\right)\right)\left(\frac{f(z)}{z^j} +
	O\left(\frac{\varkappa(\nu)}{\nu}\frac{M(\rho,f)}{\rho^j} \right)\right)\\=
	(\nu\alpha)^{j\alpha}\left(\frac{f(z)}z+O\left(\frac{\varkappa(\nu)}{\nu}\frac{M(\rho,f)}{\rho^j}\right)\right)\end{gather*}
when $r\rightarrow+\infty$ outside a set of finite logarithmic measure, which is (\ref{e: th1}).

We then choose $z$ in (\ref{e: th1}) in turn   to  maximise  $|f(z)|$ and $|\mathcal{D}^\alpha f(z)|$  and deduce that
$$M(\rho, \mathcal{D}^j_\alpha f)=\left(1+O\left(\frac{\varkappa(\nu)}{\nu}\right)\right)
\frac{(\nu\alpha)^{j\alpha}}{\rho^j}M(\rho,f), \quad \rho\to+\infty.$$

In order to complete the proof of (\ref{e: th2}) it is sufficient  to show that
$$\ln M(\rho,f)=\ln M(r,f)+\nu\ln(\rho/r)+o(1), \quad r\to+\infty. $$

First, we note that (\ref{e: p+r}) and (\ref{e: estfrdrt}) yield for our range of $\rho$
$$\ln M(\rho,f)=\nu_{1}\ln\rho+\ln M(\rho,P)+o(1), \quad \rho\to+\infty.$$

Then, it follows from Lemma \ref{l: lem3.5.} that
$$M(\rho,P)=M(r,P)\left(1+O\left(\frac{(\rho-r)\varkappa(\nu)}{r}\right)\right)\sim M(r,P), \quad r\to+\infty. $$
if $\varkappa(\nu)\ln(\rho/r)=o(1)$. The second equality of (\ref{e: th2}) now follows, completing the proof of Theorem  \ref{th:frwv}.



\section{$\alpha$-analyticity of soultions for \eqref{e:frac_lin_eq}} \label{s:final}
\setcounter{section}{4} \setcounter{equation}{0}

Using the Cauchy method of majorant series, we prove the existence and uniqueness of a solution to \eqref{e:frac_lin_eq}.

\begin{theorem}\label{t:Cauchy} The equation (\ref{e:frac_lin_eq}) where $p_k(x)=\sum_{m=0}^\infty p_{mk} x^{m\alpha}$, $x\in [0, \rho_k)$ are 
	$\alpha$-analytic, $k\in \{0, \dots, n-1\}$  with the initial conditions
	\begin{equation}
		\label{e:Cp}
		y(0)=b_{0},\; \mathbb{D}_\alpha y  (0)=b_{1},  \dots, \mathbb{D}_\alpha^{n-1} y (0)=b_{n-1}, \end{equation}
	has the unique $\alpha$-analytic solution $y(x)=\sum_{m=0}^\infty a_{m} x^{m\alpha}$, $x\in [0, \rho)$, where $\rho =\min \{\rho_0, \dots, \rho_{n-1} \} $.
\end{theorem}
\textbf{Proof of Theorem \ref{t:Cauchy}.} 
In our notation we have the  representation for the fractional derivative  of a formal solution due to Remark \ref{r:2}
\begin{equation}\label{e:k_deriv}
	\mathbb{D}_\alpha^k y (x)=\sum_{m=0}^\infty a_{m+k} \frac{\Gamma((m+k)\alpha+1)}{\Gamma(m\alpha+1)} x^{m\alpha}.
\end{equation}
This yields $\mathbb{D}_\alpha^k y(0)=a_k \Gamma(k\alpha+1)$. Hence, $a_k=b_k/\Gamma(k\alpha+1)$, $k
\in \{0, \dots, n-1\}$. 
Substituting \eqref{e:k_deriv} into \eqref{e:frac_lin_eq} we obtain
\begin{gather} \nonumber 
	\sum_{m=0}^\infty a_{m+n} \frac{\Gamma((m+n)\alpha+1)}{\Gamma(m\alpha+1)} x^{m \alpha}\\ \nonumber
	=- \sum_{k=0}^{n-1} \left( \sum_{m=0}^{\infty} p_{mk} x^{m\alpha} 
	\sum_{m=0}^\infty a_{m+k} \frac {\Gamma((m+k)\alpha+1)}{\Gamma(m\alpha+1)} x^{m\alpha}  \right)\\ =
	-\sum_{k=0}^{n-1} x^{m\alpha}\sum_{s=0}^m a_{s+k} \frac {\Gamma((s+k)\alpha+1)}{\Gamma(s\alpha+1)} p_{m-s,k}.
	\label{e:subst_ser}
\end{gather}

Equating the coefficients of the same degree in (\ref{e:subst_ser}), we write
\begin{equation} \label{e:reccur}
	a_{m+n} \frac{\Gamma((m+n)\alpha+1)}{\Gamma(m\alpha+1)}=-\sum_{k=0}^{n-1} \sum_{s=0}^m a_{s+k} p_{m-s,k} \frac{\Gamma((s+k)\alpha+1)}{\Gamma(s\alpha+1)}, \quad m\in \ZZ_+.
\end{equation}
Let $r\in (0, \rho)$. Then there exists $M>0$ such that 
\begin{equation} \label{e:p_k_est}
	|p_{j,k}|\le \frac{M}{r^{j\alpha}}, \quad j\in \ZZ_+, \; k\in \{0, 1, \dots, n-1\}.
\end{equation}
\begin{lemma}\label{l:coef_est}
	Under the above conditions the following estimate for the coefficients is valid
	\begin{equation} \label{e:coeff_est}
		|a_p|\le \beta_p \prod_{j=1}^{p-1} \left( \frac 1{r^\alpha} +\beta_j\right) 
	\end{equation} 
	for some positive sequence $(\beta_p)$ such that $\beta_p\to 0$ as $p\to \infty$, where $\prod\limits_{j\in \varnothing}c_j:=1$.
	
	In particular, $\lim\limits_{p\to \infty}\sqrt[p]{|a_p|}\le r^{-\alpha}$.
\end{lemma}
\begin{proof}[Proof of Lemma \ref{l:coef_est}]
	It follows from  properties of the Gamma function that 
	\begin{equation}\label{e:gamma_m}  \frac{\Gamma(m\alpha+1)}{\Gamma((m+n)\alpha+1)}\le \frac{M_1}{(m\alpha+1)^{n\alpha}}=:\gamma_m, \quad m\in\ZZ_+, n\in \NN
	\end{equation}
	and 
	\begin{equation}\label{e:delta}  \frac{\Gamma((s+k)\alpha+1)}{\Gamma(s\alpha+1)}\le {M_2}{((s+k)\alpha+1)^{k\alpha}}=:\delta_{s,k}, \quad s,k\in\ZZ_+, n\in \NN.
	\end{equation}
	The values $\beta_0 $, \dots, $\beta_{n-1}$ are chosen recursively  so that the equality in \eqref{e:coeff_est} holds, i.e. 
	$$\beta_0=|a_0|,\; \beta_1=|a_1|, \; |a_2|=\beta_2(r^{-\alpha}+\beta_1),  \dots, |a_{n-1}|=\beta_{n-1}\prod_{j=1}^{n-2} \left( \frac 1{r^\alpha} +\beta_j\right).$$ 
	We then write $$ \beta_{m+n}=\frac{nMM_1M_2 C}{(m\alpha+1)^\alpha}/{\min_{0\le k\le n-1} \prod_{j=m+k+1}^{m+n-1} \left( \frac 1{r^\alpha} +\beta_j\right)} , \quad m\ge 0,$$
	where the constant $C$ will be specified later. It is an elementary exercise  to prove that $\beta_p\to 0$ as $p\to \infty$.
	
	We prove \eqref{e:coeff_est} by induction. Since the induction base holds by the choice of $\beta_0, \dots, \beta_{n-1}$, it is  sufficient  to prove the induction step.
	
	Let $m\ge 0$, and \eqref{e:coeff_est} hold for $0\le p\le m+n-1$. Then by \eqref{e:reccur} 
	\begin{gather} \nonumber
		|a_{m+n}|\le \gamma_m \sum_{k=0}^{n-1} \sum_{s=0}^m  \beta _{k+s}\prod_{j=1}^{k+s-1} \left( \frac 1{r^\alpha} +\beta_j\right) \frac M{r^{\alpha(m-s)}} \delta_{s,k}\\
		\le \gamma_m M \sum_{k=0}^{n-1} \delta_{m,k} \sum_{s=0}^m  \frac {\beta _{k+s}}{r^{\alpha(m-s)}}  \prod_{j=1}^{k+s-1} \left( \frac 1{r^\alpha} +\beta_j\right) .	\label{e:a_m_n_est}
	\end{gather}
	
	Consider the internal sum. We have 
	\begin{gather*}
		\sum_{s=0}^m  \frac {\beta _{k+s}}{r^{\alpha(m-s)}}  \prod_{j=1}^{k+s-1} \left( \frac 1{r^\alpha} +\beta_j\right) 
		= \prod_{j=1}^{k-1} \left( \frac 1{r^\alpha} +\beta_j\right) \\ \times \left(\frac{\beta_k}{r^{\alpha m}}+ \frac{\beta_{k+1}}{r^{\alpha (m-1)}}\left(\frac 1{r^\alpha}+\beta_k\right)+\dots + \beta_{k+m} \prod_{j=k}^{k+m-1} \left( \frac 1{r^\alpha} +\beta_j\right) \right)\\
		< \prod_{j=1}^{k-1} \left( \frac 1{r^\alpha} +\beta_j\right) \left(\frac{1}{r^{\alpha (m+1)}}+ \frac{\beta_k}{r^{\alpha m}}+ \frac{\beta_{k+1}}{r^{\alpha (m-1)}}\left(\frac 1{r^\alpha}+\beta_k\right)+\dots\right. \\ \left. + \beta_{k+m} \prod_{j=k}^{k+m-1} \left( \frac 1{r^\alpha} +\beta_j\right) \right)\\
		= \prod_{j=1}^{k-1} \left( \frac 1{r^\alpha} +\beta_j\right) \prod_{j=k}^{k+m} \left( \frac 1{r^\alpha} +\beta_j\right) =\prod_{j=1}^{k+m} \left( \frac 1{r^\alpha} +\beta_j\right) .
	\end{gather*}
	Substituting this estimate into \eqref{e:a_m_n_est} we get
	\begin{gather*}
		|a_{m+n}|\le \gamma_m M \sum_{k=0}^{n-1} \delta_{m,k} \prod_{j=1}^{k+m} \left( \frac 1{r^\alpha} +\beta_j\right) \\ \le 
	M	M_1M_2  \sum_{k=0}^{n-1} \frac{((m+k)\alpha+1)^{k\alpha}}{(m\alpha+1)^{n\alpha}} \prod_{j=1}^{k+m} \left( \frac 1{r^\alpha} +\beta_j\right) \\
		\le MM_1M_2  \max_{0\le k\le n-1} \prod_{j=1}^{k+m} \left( \frac 1{r^\alpha} +\beta_j\right) \sum_{k=0}^{n-1} \frac{C}{(m\alpha+1)^{(n-k)\alpha}}\\
		\le \frac{nMM_1M_2C}{(m\alpha+1)^\alpha} \prod_{j=1}^{m} \left( \frac 1{r^\alpha} +\beta_j\right) \max_{0\le k\le n-1} \prod_{j=m+1}^{k+m} \left( \frac 1{r^\alpha} +\beta_j\right)\\
		= \frac{nMM_1M_2C}{(m\alpha+1)^\alpha}\frac{\prod_{j=1}^{m+n-1} \left( \frac 1{r^\alpha} +\beta_j\right)}{\min_{0\le k\le n-1} \prod_{j=m+k+1}^{m+n-1} \left( \frac 1{r^\alpha} +\beta_j\right)}=\beta_{m+n} \prod_{j=1}^{m+n-1} \left( \frac 1{r^\alpha} +\beta_j\right),
	\end{gather*}
	where $C=\sup_{m\in \ZZ_+} \left(\frac{(m+n-1)\alpha+1}{m\alpha+1}\right)^{n-1}$.
	Since the product in the denominator is uniformly in $m$ bounded from above and separated from zero, the induction step is proved. 
\end{proof}


\begin{theorem}\label{t:sol_finite_order} 
	If all coefficients $P_j$ of \eqref{e:frac_lin_eq} are polynomilals, then all $\alpha$-analytic soultions have the form $v(t^\alpha)$ where 	$v$ is an entire function of finite order of the growth.
\end{theorem}

\begin{proof}[Proof of Theorem \ref{t:sol_finite_order}]
	By Theorem \ref{t:Cauchy}, there is an entire function $v$ such that $y(t)=v(t^\alpha)$ is the unique solution to the Cauchy problem \eqref{e:frac_lin_eq}, \eqref{e:Cp}. By Theorem \ref{th:mrwv}, there exists a set $E\subset [1, \infty)$ of finite logarithmic measure  such that
	\begin{equation}\label{e:wv_rel} 
		\mathcal{D}_\alpha^jv(z) =(\nu(r,v)\alpha)^{j\alpha} \frac{v(z)}{z^j}(1+o(1)), \quad |z|\not \in E,
	\end{equation}
	where $z$ satisfies $M(|z|,v)=|v(z)|$.  
	
	Let $c_jz^{d_j}$ be the leading coefficient of $P_j(z)$, $j\in \{0, \dots, n-1\}$. Substituting \eqref{e:wv_rel} into \eqref{e:frac_lin_eq} and dividing by $v(z)$ we obtain $(\nu=\nu(|z|,v))$
	\begin{gather*}
		(1+o(1))\frac{(\nu \alpha)^{n\alpha}}{z^n}+ (c_{n-1}+o(1))z^{d_{n-1}}\frac{(\nu \alpha)^{(n-1)\alpha}}{z^{n-1}}+\dots\\  + (c_{1}+o(1))z^{d_{1}}\frac{(\nu \alpha)^{\alpha}}{z}+(c_{0}+o(1))z^{d_{0}}=0
	\end{gather*}	
	or 
	\begin{gather}\nonumber
		{(\nu \alpha)^{n\alpha}}+ (c_{n-1}+o(1))z^{d_{n-1}+1}{(\nu \alpha)^{(n-1)\alpha}}+\dots\\ + (c_{1}+o(1))z^{d_{1}+n-1}{(\nu \alpha)^{\alpha}}+(c_{0}+o(1))z^{d_{0}+n}=0.\label{e:puiso_mem}
	\end{gather}
	
	Considering the term $(\nu \alpha)^\alpha$ as an unknown variable, by \cite[Lemma 1.3.1]{Laine} we deduce that
	$$ (\nu \alpha)^\alpha\le 1+ \max_{0\le k\le n-1} |c_k+o(1)|r^{d_k+n-k}, \quad r\not\in E. $$
	To finish the proof of Theorem \ref{t:sol_finite_order} we need one more lemma.
	\begin{lemma}[{\cite[Lemma 1.1.2]{Laine}}]\label{l: expset} Let $g \colon(0,+\infty)\rightarrow\mathbb{R}$, $h\colon (0,+\infty)\rightarrow\mathbb{R}$ be monotone increasing functions such that $g(r)\leq h(r)$ outside  an exceptional set $E$ of finite logarithmic measure. Then, for any $\gamma >1$, there exists $r_{0}>0$ such that $g(r)\leq h(r^{\gamma})$ holds for all $r>r_{0}$.
	\end{lemma}
	Applying this lemma we deduce that $\nu(r,v)=O(r^\sigma)$ as $r\to \infty$, where $\sigma > \frac 1\alpha \max_{0\le k\le n-1} (d_k+n-k)$. As a consequence, the order  $\sigma(v)$ does not exceed this number. The theorem is proved.
\end{proof}

\begin{theorem}\label{t:sharp_order}
	Let $P_j$ be polynomials of degree $d_j=\deg P_j$, $j\in \{0, \dots, n-1\}$, $p_0\not\equiv 0$, and $\max_{0\le k\le n-1} \frac{d_k}{n-k}=\frac{d_0}{n}$. Then all non-trivial $\alpha$-analytic solutions $y$ of the equation \eqref{e:frac_lin_eq} has the form $y(t)=f(t^\alpha)$, $t\ge 0$, where
	the order of an entire function $f$ is $\rho(f)=\frac 1\alpha \left(1+\frac{d_0}{n} \right)$. 
\end{theorem}

\begin{cor*} Let $\alpha>0$, $n\in \mathbb{N}$,  $P$ be a nontrivial polynomial of degree $d_0$. All non-trivial $\alpha$-analytic solutions $y$ of the equation 
	$$ \mathbb{D}_\alpha^n y+ P(x^\alpha)y=0$$	 has the form $y(t)=f(t^\alpha)$, $t\ge 0$, where
	the order of an entire function $f$ is $\rho(f)=\frac 1\alpha \left(1+\frac{d_0}{n} \right)$.
	\end{cor*}
\begin{proof}[Proof of Theorem \ref{t:sharp_order}] 
	First, we show that $$\sigma(f)\le \sigma_0:=\frac 1\alpha \max_{0\le k\le n-1}\left( \frac{d_k}{n-k}+1\right).$$
	Suppose the  contrary. Then, by \eqref{e:order} there exists $\eta>0$ and a sequence of positive numbers $(r_n)$ tending to $+\infty$ such that $1<r_m<r_{m+1}/2$ with $\nu(r_m)\ge r_m^{\sigma_0+\eta}$. Let $F=\bigcup_{m=1}^\infty [r_m, 2r_{m}]$. Clearly,   $F$ has infinite logarithmic measure. Moreover, for $r\in F$, we have that $r\in [r_m, 2r_{m}]$ for some $m=m(r)$. Since $\nu(r)$ is non-decreasing, 
	\begin{equation}\label{e:nu_log_mes}
		\nu(r)\ge \nu(r_m)\ge r_{m}^{\sigma_0+\eta}\ge \frac{r^{\sigma_0+\eta}}{2^{\sigma_0+\eta}}, \quad r\in F.
	\end{equation}
	Therefore, for $r\in F\setminus E$, which is of infinite logarithmic measure, and, in particular, unbounded, we have  that  \eqref{e:puiso_mem} holds. Note that for every $j\in \{1, \dots, n-1\}$ and $\varepsilon>0$ the following estimates are valid
	\begin{gather*}
		(c_j+o(1))|z|^{d_j+n-j} (\nu \alpha)^{\alpha j}\le (c_j \alpha^{j\alpha}+o(1))r^{d_j+n-j+\alpha j(\sigma_0+\varepsilon)}\\ \le  (c_j \alpha^{j\alpha}+o(1))r^{\frac{\alpha(n-j)}{\alpha}\left(\frac{d_j}{n-j}+1\right)+\alpha j(\sigma_0+\varepsilon)}
		\le (c_j \alpha^{j\alpha}+o(1)) r^{\alpha \sigma_0 \left(n+\frac{\varepsilon}{\sigma_0 j}\right)}.
	\end{gather*}
	That is, \eqref{e:puiso_mem} becomes 
	$$ (\nu \alpha)^{n\alpha} +O(r^{\alpha \sigma_0 \left(n+\frac{\varepsilon}{\sigma_0 j}\right)})=0, \quad r\in F\setminus E,$$
	which contradicts \eqref{e:nu_log_mes} provided that $\varepsilon \in (0, \eta n)$.  Thus, $\sigma(f)\le \sigma_0$.
	
	We now prove the converse inequality. It follows directly from \cite[Lemma 4.2]{GSW} that
	$$ n-k+d_k+k s < d_0+n s, \quad k\in\{1, \dots, n\},$$
	where $\frac{d_0}{n}=\max_{0\le k\le n-1}\frac{d_k}{n-k}$ for any real $s< \sigma_0 \alpha$. That is $\nu(r)=O(r^\sigma)$, $\sigma<\sigma_0$ is also impossible, because in this case  \eqref{e:puiso_mem} can be rewritten in the form 
	$ (c_0+o(1))z^{d_0}=0$. The theorem is proved.		
\end{proof}

{\it Address 1:} Faculty of Mechanics and Mathematics, Lviv Ivan Franko National University,
Universytets'ka 1, 79000,
 Lviv,  Ukraine\\
 
 Faculty of Mathematics and Computer Sciences, University of Warmia and Mazury in Olsztyn, S\l oneczna 54, 10-710 Olsztyn, Poland
 \\

{\it e-mail:} chyzhykov@yahoo.com

%

\end{document}